\documentclass[12 pt]{amsart}

\newtheorem{theorem}{Theorem}[section]
\newtheorem{lemma}[theorem]{Lemma}
\newtheorem{proposition}[theorem]{Proposition}
\newtheorem{corollary}[theorem]{Corollary}

\theoremstyle{definition}

\newtheorem{remark}[theorem]{Remark}

\theoremstyle{remark}

\usepackage{amscd,amssymb}

\begin{document}
\baselineskip 15.5 pt

\title[On the lifting of the Nagata automorphism]
{On the lifting of the Nagata automorphism}

\author[A.Belov-Kanel and Jie-Tai Yu]
{Alexei Belov-Kanel and Jie-Tai Yu}
\address{Department of Mathematics,
Bar-Ilan University Ramat-Gan, 52900 Israel} \email{beloval@cs.biu.ac.il,\ kanelster@gmail.com}
\address{Department of Mathematics, The University of Hong
Kong, Hong Kong SAR, China} \email{yujt@hkusua.hku.hk,\
yujietai@yahoo.com}


\thanks{The research of Jie-Tai Yu was partially
supported by an RGC-GRF Grant.}



\maketitle

\noindent $\mathbf{Abstract.}$\ It is proved that
the Nagata automorphism (Nagata coordinates, respectively) of the polynomial algebra $F[x,y,z]$ over a field $F$
 cannot be lifted to a
$z$-automorphism ($z$-coordinate, respectively) of the free associative algebra $F\langle x,y,z\rangle$. The proof is based on the following two new results which have their
own interests: degree estimate of ${Q*_FF\langle x_1,\dots,x_n\rangle}$ and tameness of the automorphism group
${\text{Aut}_Q(Q*_FF\langle x,y\rangle)}$.

\smallskip

\section{\bf Introduction and main results}

\smallskip

\noindent The long-standing famous Nagata conjecture for characteristic $0$ was proved
by Shestakov and Umirbaev \cite{SU1, SU2}, and a strong version
of the Nagata conjecture was proved by Umirbaev and Yu \cite{UY}.
That is, the Nagata automorphism
$(x-2y(y^2+xz)-(y^2+xz)^2z, y+(y^2+xz)z, z)$ (Nagata coordinates $x-2y(y^2+xz)-(y^2+xz)^2z$ and $y+(y^2+xz)z$ respectively)
 is (are) wild. In \cite{MSY, UY}, a stronger question (which implies the Nagata conjecture
and the strong Nagata conjecture) was raised:
whether the Nagata automorphism (coordinates) of the polynomial algebra $F[x,y,z]$ can be lifted to an automorphism
(coordinates) of the free associative $F\langle x,y,z\rangle$ over a field $F$? We can also formulate

\

\noindent {\bf The General Lifting Problem.}
Let $\phi=(f_1,\dots, f_n)$ be an automorphism of the polynomial algebra
$F[x_1,\dots,x_n]$ over a field $F$. Does there exists an $F$-automorphism
$\phi^{\prime}=(f_1^{\prime},\dots, f_n^{\prime})$ of the free associative algebra
$F\langle x_1,\dots, x_n\rangle$ such that each $f_i$ is the abelianization
of $f_i^{\prime}$?

\

\noindent For $n=2$, the answer of the above problem is positive,
as due to Jung \cite{J} and van der Kulk \cite{vdK}  every automorphism of $F[x,y]$ is
composation of linear and elementary automorphisms which are liftable
to automorphisms of $F\langle x,y\rangle$. Moreover, Makar-Limanov \cite{ML} and Czerniakiewicz \cite{Cz} proved independently that Aut$(F\langle x,y\rangle)$ is actually isomorphic to  Aut$(F[x,y])$, which implies
that the lifting is unique.



\

\noindent In this paper we prove the following new result,
which partially answers the question raised in \cite{UY} negatively. The result can be viewed
as the first  step to attack the general lifting problem. In
a forthcoming paper \cite{BY}, we will deal the general lifting problem.

\begin{theorem} \label{lifting}
Let $(f,\ g)$ be an wild $F[z]$-automorphism of $F[x,y,z]=F[z][x,y]$.
Then $(f,g,z)$, as
an $F$-automorphism of $F[x,y,z]$,  cannot be lifted to an
automorphism of $F\langle x,y,z\rangle$ fixing $z$.
\end{theorem}


\noindent The crucial step to prove Theorem \ref{lifting} is the following

\begin{theorem}     \label{main}
Let $(f,g)$ be a wild
$F[z]$-automorphism of $F[x,y,z]=F[z][x,y]$, which can be effectively obtained
as the product of the canonical sequence of uniquely determined
alternative operations (elementary $F(z)$-automorphisms), and the sequence contains an
elementary $F(z)$-automorphism of
the type $(x, y+z^{-k}x^l+\dots)$ or
$(x+z^{-k}y^l+\dots, y)$ where $l>1$. Then $(f,g,z)$, as
an $F$-automorphism of $F[x,y,z]$,  cannot be lifted to an
automorphism of $F\langle x,y,z\rangle$ fixing $z$.
\end{theorem}

\begin{corollary}
The Nagata automorphism cannot be lifted to an
automorphism of $F\langle x,y,z\rangle$ fixing $z$.
\end{corollary}

\begin{corollary}
Let $(f,g)$ be a wild
$F[z]$-automorphism of $F[x,y,z]=F[z][x,y]$.
Then neither $f$ nor $g$ can be lifted
to a $z$-coordinate of $F\langle x,y,z\rangle$. In particular,
the Nagata coordinates $x-2y(y^2+xz)-(y^2+xz)^2z$ and  $y+(y^2+xz)z$
cannot be lifted to any $z$-coordinate of $F\langle x,y,z\rangle$.
\end{corollary}
\begin{proof}
Suppose $(f,h)$ is an $F[z]$-automorphism, then obviously
$(f,h)$ is the product of $(f,g)$ and an elementary
$F[z]$-automorphism of the type $(x,h_1)$. Therefore
$(f,h)$ is liftable if and only if $(f,g)$ is liftable.
Hence any $F[z]$-automorphism of the type $(f,h)$
is not liftable. Therefore $f$ cannot be lifted to
a $z$-coordinate of $F\langle x,y,z\rangle$. Same for $g$.
\end{proof}

\noindent Crucial to the proof of Theorem \ref{main} is the following new result,
that implies that the automorphism group
${\text{Aut}_Q(Q*_FF\langle x,y\rangle)}$ is tame, which has
its own interests.

\begin{theorem}[on degree increasing process]   \label{degreeincrease}
Let  $Q$ be an extension field over a field $F$.
A $Q$-automorphism of $Q*_FF\langle x,y\rangle$ can be effectively obtained as the
product of a sequence of
uniquely determined alternating operations (elementary automorphisms) of the following types:

\begin{itemize}
    \item  $\quad x\to x,\ y\to ryr'+\sum r_{0}xr_{1}x\cdots
r_kxr_{k+1},$
    \item $x\to qxq'+q_{0}\sum yq_{1}y\cdots q_kyq_{k+1},\quad y\to y$
\end{itemize}
where $r,q, r_{j}, q_{j}\in Q$.
\end{theorem}

\noindent The following new result of degree estimate  is also essential to the proof
of Theorem \ref{main}.

\begin{theorem}[Degree estimate]    \label{degreeestimate}
Let  $Q$ be an extension field of a field $F$.
Let $A = Q*_FF\langle x_1,\dots, x_n\rangle$ be a co-product
of $Q$ and the free associative algebra $F\langle x_1,\dots, x_n\rangle$ over $F$.
Suppose  $f,g\in A$ are algebraically independent over $Q$,
$f^+$ and $g^+$ are algebraically independent over $Q$; or $f^+$ and $g^+$
are algebraically dependent,  and neither $f^+$ is
$Q$-proportional to a power $g^+$, nor $g^+$ is
$Q$-proportional to a power $f^+$.  Let $P \in
Q*_FF\langle x,y\rangle\backslash Q$.
Then
$$\deg(P(f,g))
\geq
\frac{\deg([f,\ g])}{\deg(fg)}w_{\deg(f), \deg(g)}(P),$$
where the degree is the usual homogeneous degree
with respect to $x_1,\dots,x_n$ and $w_{r,s}$ is the weight degree
with respect to $r,s.$\end{theorem}

\noindent Note that $u$ is {\bf proportional} to $v$ for $u,\ v\in Q*_FF\langle x_1,\dots, x_n\rangle$
means that there exist $p_1,\dots, p_m; q_1,\dots, q_m\in Q$ such that
$u=\Sigma_{i=1}^m p_ivq_i$ (it is important that `proportional'
is not reflexive, i.e. $u$ is proportional to $v$
does not imply $v$ is proportional to $u$), and that $f^+$ is the highest homogeneous form of $f$.

\

\begin{remark}
Theorem \ref{degreeestimate} is still valid for an arbitrary
division ring $Q$ over a field $F$. The proof is almost the same.
When $Q=F(z)$ then the result can be directly deduced from the
degree estimate in \cite{MLY, YuYungChang} via substitution
$\psi:\ x\to P_1(z)xP_2(z); y\to R_1(z)yR_2(z)$ for appropriate
$P_i, R_i\in F[z]$. For any element $\tau$ in $Q*_FF\langle
x,y\rangle$ there exist $P_i, R_i\in F[z]$ such that
$\psi(\tau)\in F\langle x,y,z\rangle$. In the sequel we only use
this special case regarding the lifting problem.
\end{remark}




\section{Proofs}

\noindent{\bf Proof of Theorem \ref{degreeestimate}}
Similar to the proof of the main result of Li and Yu
\cite{YuYungChang}, where Bergman's Lemma \cite{B1, B2} on
centralizers is used. See also Makar-Limanov and Yu \cite{MLY} for
the special case of characteristic $0$, where Bergman's Lemma on
radical \cite{B1, B2} is used.\qed

\

\noindent {\bf Proof of Theorem \ref{degreeincrease}.} Let $\phi=(f,g)$
be a  $Q$-automorphism in $\text{Aut}_Q(Q*_FF\langle x,y\rangle)$ which is not linear,
namely,
$$\deg_{x,y}(f)+\deg_{x,y}(g)\ge 3.$$
By Theorem 1.5, we obtain
that either a power of $f^+$ is proportional to $g^+$,
or a power of $g^+$ is proportional to $f^+$. Now the proof is done by
induction.\qed

\

\noindent To prove the  main result, we need a few more
lemmas.

\

\noindent {\bf Definition.} Let $D$ be a  domain containing a field $K$,\ $E$
the field of fractions of $D$.
A monomial $\in E*_KK\langle
x, y\rangle$ of the following form,
$$...... {ptq} ......\ \ \ $$
 where $t\in E\backslash D$,\ $p, q\in\{x, y\}$,\
is  called a {\it sandwich monomial}, or just a {\it sandwich} for short.

\

\begin{lemma}[on sandwich preserving] \label{LeSandwich}
In the constructive decomposation in Theorem \ref{degreeincrease}, suppose
a sandwich  $... {ptq} ...$ (where $p,q\in\{x,y\}$,\ $t\in
F(z)\backslash F[z]$,  appears on some step during the process
of the effective decomposation, then there will be
some sandwich in any future step.
\end{lemma}

\begin{proof}
Let $f$ be the polynomial obtained in the  $(n-1)^{-th}$ step of
the effective operation in Theorem \ref{degreeincrease},
$k=\deg(f)$.  Take all sandwiches $s_\alpha$ of the maximum total
degree with respect to $x$ and $y$. Let $S=\sum s_\alpha$ be their
sum. Let $T=\sum t_\beta$ be the sum of components (monomials)
$t_\beta$ of $f$ maximum total degree respect to $x$ and $y$. It
is possible that  $s_\alpha=t_\beta$ for some $\alpha,\beta$, then
$\deg(s_\alpha)=\deg(t_\beta)$ for all $\alpha,\beta$. In this
case $T=S+D$.

\

\noindent Suppose the $n^{-th}$ step has the following form $x\to
x, y\to y+G(x)$. Let $\bar{G}$ be the sum of monomials in $G$ with
the maximum degree. Then $\bar{G}(x)=\widetilde{G}(x,\dots,x)$
where
$$\widetilde{G}(x_1,\dots,x_n)=\sum_iq_{i,1}x_1q_{i,2}x_2\cdots
x_mq_{i,m+1},\quad q_{ij}\in F(z),$$
\noindent $m$ be the degree
of the $n^{-th}$ step operation (elementary automorphism).

\

\noindent Let $\deg(S)<\deg(T)$. Consider elements of the form
$\widetilde{G}(S,T,\dots,T)$. It is a linear combination of
sandwiches. All of them have the following form
$$q_0s_iq_1t_{2}\cdots
t_{m-1}q_m,\quad q_i\in Q.$$
Their sum is not zero, because for any  polynomial of the form
$H=\sum_iq_{i_1}x_1q_{i_2}x_2\cdots x_mq_{i_{m+1}}$ such that
$H(x,\dots,x)\ne 0$ and for any $S, T\notin Q$,\ \
$H(S,T,\dots,T)\ne 0$.

\

\noindent If $\deg(S)=\deg(T)$, we consider elements of the form
$\widetilde{G}(S,S,\dots,S)$. It is a linear combination of
sandwiches. All of them have the following form
$$q_0s_1q_1\cdots
s_{m}q_m,\quad q_i\in Q.$$ \noindent $s_i$ are monomials from $S$.
Their sum is not zero, because for any  polynomial of the form
$H=\sum_iq_{i_1}x_1q_{i_2}x_2\cdots x_mq_{i_{m+1}}$ such that
$H(x,\dots,x)\ne 0$ and for any $S\in Q$,\ \ $H(S,\dots,S)\ne 0$.

\

\noindent Now we are going to prove (via degree estimate) that
they cannot cancel out by other monomials (which must be
sandwiches). That is, there are no other sandwiches which are in
this form. They cannot be produced by $H(R_1,\dots,R_m)$ where
$R_i$ are monomials either from $S$ or $T$. This can be easily
seen from the following argument:

\

\noindent Suppose $\deg(S)=\deg(T)$ and $D\ne 0$. Then
if we substitute monomials forming $S$ and $D$ in different
`words', the outcomes would be different. Similarly suppose $\deg(S)<\deg(T)$ and $D\ne 0$,
then substituting monomials forming $S$ and $T$ in different `words',
the outcomes would be different. Suppose $D=0$, i.e. $S=T$. Then
$H(T,\dots,T)=H(S,\dots,S)$. Obviously in this case we need to do nothing.

\

\noindent Suppose we get  such a sandwich because an element from
$F(z)\backslash F[z]$ appears between summands of polynomials,
obtained  by the the (previous) $(n-1)^{-th}$ step. It means that
`fractional coefficient' in $F(z)\backslash F[z]$ appears in the
position in some term,  between two monomials obtained on the
$(n-1)$-th step. Let us describe this situation in more details.

\

\noindent Let
$$x\to \sum v_i,\quad y\to \sum u_i$$
be an
automorphism, obtained on the $(n-1)$-th step. Consider $n$-th
step:
$$x\to x,\ y\to y+\sum_i q_0^ixq_1^i\cdots xq_{n_i}^i.$$
Let
$v_i=a_i\bar{v}_ib_i$ where $a_i, b_i\in Q, \bar{v}_i$ begins with
either $x$ or $y$ and  also ends with either $x$ or $y$.

\

\noindent Suppose the leftmost factor $q\in F(z)\backslash F[z]$ corresponding
to the leftmost factor $q$ in monomial $s_i$ in $s$ appears in the
corresponding sandwich $w$.
Then it has the form
$$w=q_0^iv_{\alpha_1}q_1^i\cdots a_{\alpha_k}\bar{v}_{\alpha_k}b_{\alpha_k}q_k^i
a_{\alpha_{k+1}}\bar{v}_{\alpha_{k+1}}b_{\alpha_{k+1}}\cdots
a_{\alpha_{n_i}}\bar{v}_{\alpha_{n_i}}b_{\alpha_{n_i}}q_{n_i}^i$$
and the position $\bar{v}_{\alpha_k}b_{\alpha_k}q_k^i
a_{\alpha_{k+1}}\bar{v}_{\alpha_{k+1}}$ corresponds to the
position of fractional coefficient in the sandwich $s\cdot
\prod_{i=1}^{n-1}v_i$ living inside $s=s_1qs_2$, $s_1$ ends with
$x$ or $y$, $s_2$ begins with $x$ or $y$. Then
$$q=b_{\alpha_k}q_k^i a_{\alpha_{k+1}},
s_1= q_0^iv_{\alpha_1}q_1^i\cdots a_{\alpha_k}\bar{v}_{\alpha_k},
s_2T_{i_2}\cdots T_{i_n}=$$
$$=\bar{v}_{\alpha_{k+1}}b_{\alpha_{k+1}}\cdots
a_{\alpha_{n_i}}\bar{v}_{\alpha_{n_i}}b_{\alpha_{n_i}}q_{n_i}^i.
$$
Only in that case cancellation is possible. Here $T_i$ are
monomial summands of $T$.

\

\noindent Now let us compare the degrees. $\deg(s_1)<\deg(s)$,
$\sum_{i=k}^n\deg(v_i)\le (n_i-k+1)\deg(T)\le (m-1)\deg(T)$. Hence
$$\deg(W)<\deg(s)+(m-1)\deg(T)=\deg(\widetilde{G}(s,T,\dots,T))$$ so any
cancellation is impossible.
%
%
%
%
\end{proof}


\begin{lemma}[on coefficient improving]       \label{LeCoeffImpr}
\noindent a) Let $x'=pxq;\ p,q\in F(z)$, $M_{\vec{q}}(x) =xq_1xq_2\cdots x$,
$q_i'=q^{-1}q_ip^{-1}$. Then $M_{\vec{q}}(x')= pM_{\vec{q'}}(x)q$.

\noindent b) Take the process in Theorem \ref{degreeincrease} without
sandwiches. Then after each step (except the last step), the outcome $(f,g)$
has the following properties:

\begin{itemize}
    \item  Both $f$ and $g$ are sandwich-free.
    \item The left coefficients of $f$ and $g$  belong to $F[z]$. Moreover,
    the two coefficients are relatively prime.
    \item The right coefficients  of $f$ and $g$ belong to $F[z]$.  Moreover,
    the two coefficients are relatively prime.
\end{itemize}

\noindent Now we can clearly see that the outcome of the last step also has  the
above property.
\end{lemma}

\begin{proof}  a) is obvious;\ b) is a consequence of  a).\end{proof}

\

\begin{lemma}    \label{Lefinish}
Let $f\in F(z)*_FF\langle x,y\rangle$,\ $P(u)\in F(z)*F[u]$ such
that each monomial has degree $\ge 2$ respect to $x$ and $y$.
Suppose that one of the coefficients of $P$ has zero right
$z$-degree and one of the coefficients of $f$ has zero right
$z$-degree, and there is no coefficients of $P$ and $f$ with
negative right $z$-degree. Then $P(f)$ has one of the coefficients
with zero right $z$-degree and the degree (respect to $x$ and $y$)
of corresponding term is strictly more then $\deg(f)$.
\end{lemma}

\begin{proof} Consider the highest degree monomials of $P$ and $f$ with zero
right $z$-degree, let $\widetilde{P},\widetilde{f}$ will be their
sums. Let $g$ be sum of terms of $f$ with zero right $z$-degree,
$h$ be the sum of terms of $f$ of maximal degree.

\

\noindent Now consider again the highest degree monomials in $P(u)$ with zero
right $z$-degree and substitute $\widetilde{f}$ on the rightmost
position instead of $u$ and $h$ on other positions of $u$. We
shall get some terms with non-zero sum $T$ (same argument as in
the proof of sandwich lemma). All such terms have zero right
$z$-degree.

\

\noindent It remains to prove that such terms cannot cancel out from
the other
terms. First of all, we need to consider only terms of $P$ with
zero right $z$-degree, other terms can not make any influence.
Second, we have to consider substitutions only of terms with zero
right $z$-degree on the rightmost positions of $u$. Let $V$ be
their sum.

\

\noindent But the sum of highest terms satisfying this conditions is equal to $T$
and $T$ is the highest homogeneous component of $V$, hence $V\ne
0$.
\end{proof}

\begin{corollary}     \label{Cofinish}
Let $f$ be a polynomial, $P\in F(z)*F[x]$ such that each monomial
has degree $\ge 2$. Suppose that  one of the coefficients of
$P(f)$ has (zero)negative right $z$-degree. Then $P(f)$ has one of
the coefficients with (zero) negative right $z$-degree and degree
of corresponding term is strictly more then $\deg(f)$.
\end{corollary}

\noindent There is just the `dual' left version of lemma \ref{Lefinish} and corollary
\ref{Cofinish}.

\

\noindent As a consequence of the above corollary, we get

\begin{lemma}\label{Lm2new}
In the  step $x\to x$, ($z\to z$
because we are working with $z$-automorphisms) $y\to y+x^kz^{-l}$,\
$k>1$ of the degree-strictly-increasing process, applied to the automorphism $x\to x+\mbox{higest\ terms}, y\to
y+\mbox{higest\ terms}$ causes some negative
power(s).
\end{lemma}

\medskip

\noindent In order to prove Theorem \ref{lifting} we need a similar
statement which is also a consequence of the Corollary
\ref{Cofinish}.

\begin{lemma}\label{Lm2newLifting}
In the  step $x\to x$, ($z\to z$ because we are working with
$z$-automorphisms) $y\to y+P(x)$,\  such that $P$ has  negative
powers of $z$ as left coefficients of some monomial of degree $\ge
2$ in the degree-strictly-increasing process causes some negative
power(s) on any succeeding step.
\end{lemma}

\noindent Lemma \ref{Lefinish}, Lemma \ref{Lm2newLifting} and Corollary \ref{Cofinish} says
that any further step of non-linear operation either  contains terms
of negative power with bigger degree, or does not interfere in the
process. Hence they  imply the following


\begin{lemma}   \label{LeLinindNonCancell}
\noindent a) Consider stage in strictly increasing process of following
form.
$$
x\to T_1+h_1,\quad y\to T_2+h_2
$$
where $T_i$ are sums of the terms with negative powers of $z$ to
the right, $h_i$ -- are sums of the terms without negative powers
of $z$ to the right.

\noindent If $T_i$ are $F(z)$-linear independent, then the negative powers can
not be cancelled in the strictly increasing process.

\

\noindent b) Suppose $T_1$ is the sum of the terms with negative powers of
$z$ to the right, $h_1$ -- is the sum of the terms without
negative powers of $z$ to the right, $T_2$ is the sum of the terms
with zero powers of $z$ to the right, $h_2$ is the sum of the
terms with positive powers of $z$ to the right.

\noindent If $T_i$ are $F(z)$-linear independent, then the
negative powers can not be cancelled in the strictly increasing
process.
\end{lemma}


\noindent In order to prove Theorem \ref{lifting} we need slight
generalization of the previous lemma, which also follows from the
Lemma \ref{Lefinish} and Corollary \ref{Cofinish}.

\begin{proposition}    \label{PrLinindNonCancellGen}
Consider stage in strictly increasing process of following form:
$$
x\to T_1+h_1+g_1,\quad y\to T_2+h_2+g_2
$$
where $T_i$ are sums of the terms with negative powers of $z$ to
the right, $h_i$ -- are sums of the terms with zero powers of $z$
to the right, $g_i$ are sums of the terms with positive powers of
$z$ to the right.

\

\noindent If $T_i$ are $F(z)$-linear independent, or wedge product of
vectors
$$(T_1,T_2)\bigwedge_{F[z_l,z_r]} (h_1,h_2)\ne 0,$$
then the negative powers can not be cancelled in the strictly
increasing process. Wedge product is taken respect to left and
right $F(z)$-actions, i.e. as $F[z_l,z_r]$-modula, monomial
(respect to $x, y$, and inner positions of $z$) are considered as
basis vectors.
\end{proposition}

\begin{proof} Pbviously, any linear operation cannot cancel the negative powers of
$z$, but Lemma \ref{Lefinish} and corollary \ref{Cofinish} allows us
to consider only such operations.
\end{proof}

\begin{remark}         \label{RemValuations}
Considering the substitutions $z\to z+c$ one can get similar results
for negative powers of $z+c$ (or via considering other valuations
of $F(z)$).
\end{remark}

\begin{lemma}    \label{Leverylast}
Consider the step in the strictly increasing process of following form.
$$
x\to T+h_1',\quad y\to U+h_2'
$$
where $T$ is  the sum of the terms with negative powers of $z$ to
the right, $U$ the sum of the terms with negative powers of $z$ to
the right,  $h_1$ is the sum of the terms without negative powers
of $z$ to the right, $h_2$ is the sum of the terms with positive
powers of $z$ to the right.

\

\noindent If $T$ and $U$ are $F(z)$-linear independent, then the negative
powers cannot be cancelled  in the strictly increasing process.
\end{lemma}

\begin{proof} By induction. Input of composition with polynomials
with $x$-degree $\ge 2$  cannot be cancelled (otherwise some negative power
appears in the highest terms, and the $F(z)$-independence preserves).
But the $x$-linear term action  only produces the $F(z)$-linear combinations.
\end{proof}

\noindent Consider, for instance, the elementary automorphism $$x\to x,\quad y\to
y+z^nx^k.$$ It can be lifted to an $Q$-automorphism $$x\to x,\quad
y\to y+z^{n_0}x^{k_1}z^{n_1}\cdots x^{k_s}z^{n_s},\quad \sum
k_i=k,\ \sum n_i=n.$$ Though $n<0$,  $n_0$ and $n_s$ can still be
non-negative. It is necessary to deal with that kind of situation by the next lemma.

\begin{lemma}\label{LeFinishSandwich}
Consider a elementary mapping
$$x\to x;\quad y\to y+P(x)$$
such that $P(x)$ has a monomial of the following form:
$$z^{k_1}xz^{k_2}x\cdots xz^{k_s}$$
where one of $k_i<0$ for some $i$ such that $1<i<s$. Then if such an
elementary transformation occurs in the strictly increasing process, it must produce some
sandwich.
\end{lemma}
\begin{proof}
First of all, due to the Lemma \ref{LeSandwich}, we may assume without
loss of generality that there exists no sandwiches before this
step.

\

\noindent Consider $z^{k_i}$, the minimum power of $z$, lying
before the variables for all monomials in $P$. Next, consider the
monomials in $P$ of the minimum degree containing $z^{k_i}$ between
$x$'s and among them, i.e. the monomials such that $z^{k_i}$
positioned on the left-most possible position (but then
$i>1$, it should be a sandwich position). Let us denote such
terms $T_i$.

\

\noindent Let $$\varphi(x)=\sum u_i,\quad \varphi(y)=\sum v_i$$
will be an automorphism, obtained by the previous step. Due to the
Lemma \ref{LeSandwich} we may assume that no terms come from $u_i, v_i$
are sandwiches.

\

\noindent Now we consider $u_i$ with minimal right $z$-degree
$n_r$, and among them -- terms with minimal degree (respect to $x$
and $y$). Let $u_j^r$ will be such terms, $u^r=\sum u_j^r$.
Because $x$ is one of the $u_i$, $n_r\le 0$. Similarly we consider
$u_i$ with minimal left $z$-degree $n_l$, terms $u_j^l$ and their
sum $u^l=\sum u_j^l$. We also get $n_l\le 0$.

\

\noindent Now for any monomial $T_j$, consider the element
$$E_{T_j}=q^{(j)}_0x\cdots u^rz^{n_i}u^l\cdots xq^{(j)}_s$$
obtained by replacement of $u^r$ and $u^l$ into the positions of
$x$ surrounding occurrence of $z^{n_i}$  as discussed previously, the
resulting power of $z$ would be equal to $n_r+n_l+n_i\le n_i<0$.

\

\noindent Now $E_{T_j}$ can be presented as a sum $E_{T_j}=\sum M_{E_{T_j}}$,
 where $M_{E_{T_j}}$ are monomials. Monomials
$M_{E_{T_j}}$ are sandwiches, they may appear only that way which
was described previously and hence cannot  cancell by  other
monomials. Hence we must have a sandwich.
\end{proof}

\medskip

\

\newpage

\section{Proofs of the main theorems}

\noindent {\bf Proof of Theorem \ref{main}.}

\noindent Suppose the
automorphism $(f,g)$ can be lifted to a $z$-automorphism
of $F\langle x,y,z\rangle$. Then it induces an automorphism of
$F(z)*_FF\langle x,y\rangle$ and can be obtained by the  process described in
the Lemma on coefficient improving.

\

\noindent Then at some steps some negative powers of $z$ appear either
between variables or on the right or on the left and it will be
preserved to the end, due to  Lemma \ref{LeSandwich} and Lemma
\ref{LeFinishSandwich}, or Lemmas \ref{Leverylast},
\ref{LeLinindNonCancell}, \ref{Lm2new}.

\

\noindent Hence in the lifted automorphism, there exists
some negative power of $z$. A contradiction.
\qed

\

\noindent {\bf Proof of Theorem \ref{lifting}.}

\noindent Let $(f, g)$ be a wild $F[z]$-automorphism of $F[x,y,z]$ such that
it is not of the type in Theorem \ref{main}. Consider
corresponding strictly increasing process. We shall need few more
statements.

\noindent The following lemma is a consequence of Proposition \ref{PrLinindNonCancellGen}.

\begin{lemma}      \label{LeAdjoint}
In the strictly increasing process. Consider the steps with 
negative powers of $z$ appearing to the right.

$$\varphi:\ x\to x+P(y), y\to y$$


\noindent Let
$$\psi:\ x\to x, y\to y+Q_1(x)+Q_2(x)$$
where $\deg(Q_1)=1$, each term of $Q_2$ has degree $\ge 2$ and does not contain negative powers of $z$. Then
$\psi=\psi_1\circ\psi_2$ where $\psi_1:\ x\to x, y\to y+Q_1(x)$,
$\psi_2:\ x\to x, y\to y+Q_2(x)$ and $\varphi\psi_2\varphi^{-1}$
has no negative powers of $z$ to the right.
\end{lemma}

\noindent Lemma \ref{LeAdjoint} together with its left analogue and remark
\ref{RemValuations} imply following statement:

\begin{proposition}\label{PrAdjoint}
In the strictly increasing process, consider the step with
appearing coefficients not in $F[z]$.
$$\varphi:\ x\to\ x+P(y), y\to y$$
Let
$$\psi:\ x\to x, y\to y+Q_1(x)+Q_2(x)$$
where $\deg(Q_1)=1$, each term of $Q_2$ has degree $\ge 2$ and does not contain negative powers of $z$ . Then
$\psi=\psi_1\circ\psi_2$ where $\psi_1:\ x\to x, y\to y+Q_1(x)$,
$\psi_2:\ x\to x, y\to y+Q_2(x)$ and $\varphi\circ\psi_2\circ\varphi^{-1}$
is a $z$-automorphism of $F\langle x,y,z\rangle$.
\end{proposition}

\begin{proof}
Consider set of elements from $F(z)$ which are coefficients of our
monomials. If all valuations of $F(z)$ centered in finite points
are positive, then they belong to $F[z]$ and we are done. Due to
symmetry, it is enough to consider right coefficients and due to
substitution $z\to z+a$ just valuation centered in zero. Then 
by Lemma \ref{LeAdjoint}, we are done.
\end{proof}

\begin{proof}
It is easy to see that $\psi\circ\varphi\circ\psi^{-1}$ has following form: 
$x\to
x+c_1R(a'_{21}x+a'_{22}y), y\to c_2R(a'_{21}x+a'_{22}y)$,
where $a'_{ij}=\alpha a_{ij}\in F[z]$ are relatively prime, $\alpha\in F[z]$ is the least common multiple of the denominators of $a_{21}, a_{22}\in F[z]$ and $c_1, c_2\in F[z]$ such that
$c_1a_{21}+c_2a_{22}=0$. Choose $r, s\in F[z]$ such that  $ra'_{21}+sa'_{22}=1$.

\

\noindent Acting the linear automorphism $x\to rx+sy, y\to 
a'_{21}x+a'_{22}y$ over $F[z]$ to $\psi\circ\varphi\circ\psi^{-1}$, we get an automorphism of the following form: $x\to
rx+sy+tR(a'_{21}x+a'_{22}y), y\to a'_{21}x+a'_{22}y$, which is elementarily
equivalent to $x\to rx+sy, y\to a'_{21}x+a'_{22}y$. Hence $\psi\circ\varphi\circ\psi^{-1}$  is tame.
\end{proof}
\noindent The next proposition is well-known from linear algebra. 

\begin{proposition}  \label{PrLinzTame}
Let $(f,g)$ is a $z$-automorphism
of $F[z][x,y]$ linear in both $x$ and $y$.  Then
it is a  tame $z$-automorphism.
\end{proposition}

\noindent Now we are ready to complete the proof of Theorem 1.1. Suppose a
$z$-automorphism $\varphi=(f, g)$ of $F[z][x,y]$ can be lifted to an automorphism of $F[z]*_FF\langle x,y\rangle$
(i.e. an automorphism of $F\langle x,y,z\rangle$ fixing $z$),
which is decomposed into
product of elementary one according to strictly increasing
process. The coefficients of elementary operation can be in
$F(z)\backslash F[z]$ only for linear terms (see Lemma
\ref{Lm2newLifting} and Remark \ref{RemValuations}) and
conjugating non-linear elementary step with respect to the automorphisms
corresponding to these terms are $z$-tame. Hence $\varphi$ is a
product of $z$-tame automorphisms and $z$-automorphisms linear in both $x$ and
$y$. Now we are done by Proposition \ref{PrLinzTame}.

\

\noindent By carefully looking through the above proofs, we actually
obtained the following

\begin{theorem}\label{decomposation}
An automorphism $(f,g)$ in $\text{Aut}_{F[z]}F\langle x,y,z\rangle$,
can be canonically decomposed as product of the following type
of automorphisms:

\noindent i) Linear automorphisms in $\text{Aut}_{F[z]}F\langle x,y,z\rangle$;\

\noindent ii) Automorphisms which can be obtained by an
elementaty automorphism in 
$\text{Aut}_{F[z]}F\langle x,y,z\rangle$ conjugated by a linear automorphism
in 

\noindent $\text{Aut}_{F(z)}F(z)*_FF\langle x,\ y\rangle$.
\end{theorem}

\noindent Theorem \ref{decomposation} opens a way to obtain stably tameness
of $\text{Aut}_{F[z]}F\langle x,y,z\rangle$, which will be done in a separate
paper \cite{BY2}.

\

\section{\bf Acknowledgements}

\noindent Jie-Tai Yu would like to thank Shanghai University and Osaka University
for warm hospitality and stimulating atmosphere during his visit, when part
of the work was done. The authors thank Vesselin Drensky and
Leonid Makar-Limanov for their helpful comments and suggestions.

\end{document}